\documentclass[11pt,twoside]{article}
\usepackage{latexsym}
\usepackage{amssymb,amsbsy,amsmath,amsfonts,amssymb,amscd}
\usepackage{graphicx}
\usepackage{hyperref}
\usepackage{xcolor}
\newcommand{\commentout}[1]{}

\newcommand{\R}{\mathbb{R}}
\newcommand{\N}{\mathbb{N}}

\newcommand {\Chi} {{\bf \raise 2pt \hbox{$\chi$}} }

\newcommand {\ep}  {\epsilon}
\newcommand{\beq}{\begin{equation}}
\newcommand{\eeq}{\end{equation}}
\newcommand{\bea} {\begin{array}{rl}}
\newcommand{\eea} {\end{array}}
\newcommand{\bepa}{\left\{ \begin{array}{l}}
\newcommand{\eepa} {\end{array}\right.}
\newtheorem{theorem}{Theorem}[section]
\newtheorem{lemma}[theorem]{Lemma}

\newtheorem{proposition}[theorem]{Proposition}

\newcommand {\no}{\noindent}
\newcommand{\qed}{{ \hfill
                       {\unskip\kern 6pt\penalty 500 \raise -2pt\hbox{\vrule\vbox to 6pt{\hrule width 6pt
                       \vfill\hrule}\vrule} \par}   }}

\title{ {Well posedness for multi-dimensional junction problems with  Kirchoff-type conditions }}

 \author{
Pierre-Louis Lions$^{1}$ and Panagiotis Souganidis$^{2,3}$}

\date{\today}

\begin{document}

\maketitle
\pagestyle{plain}
\pagenumbering{arabic}

\begin{abstract}
\no 
\smallskip

\no We consider multi-dimensional junction problems for first- and second-order pde with  Kirchoff-type  Neumann boundary conditions and we show that  their generalized viscosity solutions are unique. It follows that any viscosity-type approximation of the junction problem converges to a unique limit. The results here are  the first of this kind and extend previous work by  the authors for one-dimensional junctions. The  proofs are  based on a careful analysis of the behavior of the viscosity solutions near the junction, including a blow-up argument that reduces the general problem to a one-dimensional one. As in our previous note, no convexity assumptions and control theoretic interpretation of the solutions are needed. 


\end{abstract}

\\
\noindent {\bf AMS Class. Numbers.} 35F21, 49L25, 35B51, 49L20.
\bigskip

\section{Notation and Terminology} 

Given $x\in \R^d$ we write $x=(x',x_d)$ with $x'\in \R^{d-1}$.  For $i=1,\ldots,K$, $\Pi_i:=\{(x',x_{d,i}) \in \R^d : x'\in \R^{d-1}, x_{d,i} \leq 0\}$ are half-planes  intersecting along the line $L:=\{(x',0): x'\in \R^{d-1}\}$ and set $\Pi:= \bigcup_{i=1}^K \Pi_i$. For simplicity   we write $x_i$ instead of $x_{d,i}.$  Given $u \in C(\overline \Pi;\R)$, if $(x',x_i)\in \overline \Pi_i$, we write $u_i(x',x_i):=u(0,\ldots, x', x_i, 0.\ldots);$ when possible, to simplify the notation,  we drop the subscript on $u_i$ and simply write $u(x', x_i).$  In this setting  
$u_{x_i}(0,x_i)$ is the exterior normal derivative of $u_i:\Pi \to \R$ on $L$. We consider  $K$-junction one dimensional problems in the domain  $I:=\bigcup_{i=1}^{K} I_i$ with junction $\{0\}$, where, for $i=1,\ldots,K$,  $I_i :=(-a_i, 0)$ and $a_i \in [-\infty, 0)$. 
We work with functions $u \in C({\overline I}; \R)$ and,  for $x=(x_1,\ldots,x_K) \in \bar I$, we write $u_i(x_i)=u(0, \ldots, x_i,\dots, 0)$;  when possible, to simplify the writing,  we drop the subscript on $u_i$ and write $u(x_i)$. 
We also use the notation $u_{x_i}$ and $u_{x_i x_i}$ for the first and second derivatives of $u_i$ in $x_i$.   
For  $w\in C(\overline I \times [0,T])$ and  $t_0 \in (0,T]$, $J^+w(0,t_0)$ and $J^-w(0,t_0)$ denote respectively the super- and sup-jets of $w$ at  $(0,t_0)$, which may be, of course, empty. If  $(p_1,\dots, p_K,a)\in J^+w(0,t_0)$, then, for all $(x,t) \in \overline I \times [0,T]$,  $w(x_i,t) \leq w(0,t_0) + p_i x_i + a(t-t_0) + \text{o}(|x| + |t-t_0|).$ If
$(p_1,\dots, p_K,a)\in J^-w(0,t_0),$ then  $w(x_i,t) \geq w(0,t_0) + p_i x_i + a(t-t_0) + \text{o}(|x| + |t-t_0|).$
\smallskip

Throughout the paper we work with viscosity sub- and super-solutions. In most cases, however, we will not be using the term viscosity. Also we will not keep repeating that $i\in \{1,\ldots, K\}$ but rather we will say for all $i$.
\section{Introduction}

\no We study the well-posedness of the generalized viscosity solutions to time dependent multi-dimensional junction problems satisfying 
a Kirchoff-type Neumann condition at the junction. We prove that the solutions satisfy a comparison principle  and, hence, are unique. It is then immediate that viscosity approximations satisfying the same boundary condition converge to the unique solution. 
Our results, which are the first of this kind,  are simple, self-contained and depend on elementary considerations about viscosity solutions and, we emphasize,  do not require any convexity assumptions and the control theoretical interpretation of the solutions. 
\smallskip

This work is a continuation of our previous paper (Lions and Souganidis \cite{LS1}) where we introduced the notion of state constraint solution to one-dimensional junction problems, proved its well-posedness, and considered, for the first time,  the limit of Kirchoff-type viscosity approximations.  
\smallskip

We also show  that the so-called flux limiter solutions introduced and studied in the references below for convex problems reduce to Kirchoff-type generalized viscosity solutions. Hence, uniqueness follows immediately by the simple arguments in this note. 
\smallskip

Among the long list of references on this topic with convex Hamiltonians we refer to Achdou and Tchou \cite{AT}, Barles, Briani and Chasseigne\cite{BBC1,  BBC2}, Barles, Briani, Chasseigne and Imbert \cite{BBCI}, Barles and Chasseigne \cite{BC}, Bressan and Hong \cite{BH}, Imbert and Monneu \cite{IM} and  Imbert and Nguen \cite{IN}. 
\smallskip

We are interested in the well-posedeness of continuous solutions $u: \overline \Pi \to \R$ to the Kirchoff -type initial boundary value problem
\begin{equation}\label{takis101}
\begin{cases}
u_{i,t} + H_i(Du_i, u_i, x,t)=0 \ \ \text{in} \ \ \Pi_i \times (0,T],\\[1mm]
\min\left(\Sigma_i u_{i, x_i}- B, \min_i (u_{i,t} + H_i(Du_i, u_i, x,t))\right) \leq 0 \ \  \text{on} \ \ L\times  (0,T],\\[1mm]
\max\left(\Sigma_i u_{i, x_i}-B, \max_i (u_{i,t} + H_i(Du_i, u_i, x,t))\right) \geq 0 \ \  \text{on} \ \ L\times  (0,T], 
\end{cases}
\end{equation}
with
\begin{equation} \label{takis20}
B\in \R \ \ \text{and}  \ \ u(\cdot, 0)=u_{0} \ \ \text{on} \ \  \overline \Pi,
\end{equation}
where 
\begin{equation}\label{takis111}
u_0\in BUC(\overline \Pi),
\end{equation}
and, for each $i$,
\begin{equation}\label{takis121}
\begin{cases}
H_i  \  \text{is \ coercive in $p$ uniformly on $x,t$ and bounded $u$, Lipshitz continuous in $u$ and $t$,}\\[1mm]
\text{and uniformly continuous in $p,u,x,t$ for bounded $p$ and $u$.}
\end{cases}
\end{equation}
As always for time-independent problems the Lipshitz continuity of $H_i$ in $u$ is replaced by 
\begin{equation}\label{takis13}
H_i  \  \text{is strictly increasing in $u$.}
\end{equation}
We remark that, as it will be clear from the proofs,  the particular choice of the Neumann condition in \eqref{takis101} 
is by no means essential. The arguments actually  apply to  more general boundary  conditions of the form $G(u_{x_1}, \ldots, u_{x_K},u),$
with the map $(p_1,\ldots, p_k, u) \to G(u_{x_1}, \ldots, u_{x_K},u)$ strictly increasing with respect to all its arguments.
\smallskip

The main result is:
\begin{theorem}\label{main1}
Assume \eqref{takis121}. 
If $u,v \in BUC(\overline \Pi \times [0,T])$ are respectively a sub-and super-solution to \eqref{takis101} with $u_i(\cdot,0)\leq v_i(\cdot,0)$ on $\overline \Pi_i$, then $u\leq v$ on $\overline \Pi \times [0,T].$
Moroever, for each  $u_0 \in BUC(\overline \Pi)$, the initial boundary value problem \eqref{takis20}, \eqref{takis111}  has a unique solution $u \in BUC(\overline \Pi \times [0,T]).$
\end{theorem}

As it will become apparent from the proof, it is possible to generalize the result to the problem 
\begin{equation}
\begin{cases}
u_{i,t} + H_i(x_i D^2u_i, Du_i, u_i, x,t)=0 \ \ \text{in} \ \ \Pi_i \times (0,T],\\[1mm]
\min \left(\Sigma_i u_{i, x_i} -B, \min_i (u_{i,t} + H_i(0,Du_i, u_i, x,t))\right) \leq 0 \ \  \text{on} \ \ L\times  (0,T],\\[1mm]
\max \left(\Sigma_i u_{i, x_i}-B, \max_i (u_{i,t} + H_i(0, Du_i, u_i,x ,t))\right) \geq 0 \ \  \text{on} \ \ L\times  (0,T], \\[1mm]
u (\cdot, 0)=u_{0} \ \ \text{on} \ \ \overline \Pi,
\end{cases}
\end{equation}
when, in addition to \eqref{takis121}, each $H_i$ is degenerate elliptic with respect to the Hessian. Since the arguments are almost identical to the ones for the proof of Theorem~\ref{main1}, we do not present  any  details.
\smallskip

Next we state the result about the convergence of  viscosity approximations to \eqref{takis101}. The claim is immediate from the fact that any limit is solution to \eqref{takis101} and, hence, we do not write the proof. We remark that we can easily use ``more complicated'' second-order approximations than the one below.
\smallskip

For $\ep>0$ consider the initial boundary value problem 
\begin{equation}\label{takis1010}
\begin{cases}
u_{i,\ep, t} - \ep\Delta u_{i,\ep}  + H_i(Du_{i, \ep}, u_{i,\ep}, x,t)=0 \ \ \text{in} \ \ \Pi_i \times (0,T],\\[1mm]
\Sigma_i u_{i,\ep,{ x_i}}=B, \ \  \text{on} \ \  L\times (0,T],\\[1mm]
u_{i,\ep}(\cdot, 0)=u_{0,i} \ \ \text{on} \ \  \overline \Pi_i,
\end{cases}
\end{equation}
which, in view of the \eqref{takis121}, if $u_0$ satisfies \eqref{takis111}  has a unique solution $u\in BUC(\overline \Pi \times [0,T]).$
\begin{theorem}\label{main2} Assume \eqref{takis111} and  \eqref{takis121}. Then $u=\lim_{\ep\to 0} u_\ep$ exists and $u$ is the unique solution to \eqref{takis101}. 
\end{theorem}
Since the proof is an immediate consequence of well known estimates and the uniqueness result, we will not discuss it any further.
\smallskip

We also show that, in the context of the one-dimensional time dependent junction problems, the flux-limiter solutions put forward in \cite{IM} are actually  generalized viscosity solution to \eqref{takis1} with appropriate choice of B in the Kirchoff condition, and, hence, are unique. This provides a simple and straightforward proof of the uniqueness without the need to consider cumbersome test functions and invoke any convexity.

\smallskip 
Following the last remark, we emphasize that Kirchoff-type conditions appear to be the ``correct'' ones, that is the only conditions that are compatible with the maximum principle. This can be easily seen, for example, at the level of second-order equations by considering affine solutions in each branch.
\smallskip

In a forthcoming paper  \cite{LS2}, we discuss problems with more general dependence on the Hessian both in the equations and along the junctions. We also consider ``stratification''-type problems, that is junctions with branches of different dimension,  
and, we  present  results about the convergence of semi-discrete in time approximations with error estimates.  Finally, we consider solutions which are not necessarily Lipschitz. 
\smallskip

In this note, to simplify the notation and explain the ideas better,  we present all the arguments in the special case $d=1$, in which case \eqref{takis101} reduces to 
\begin{equation}\label{takis1}
\begin{cases}
u_{i,t} + H_i(u_{i, x_i}, x,t)=0 \ \  \text{in} \ \ I\times(0,T],\\[1mm]
\min \left(\Sigma_i u_{i,x_i} -B, \min_i(u_{i,t} + H_i(u_{i, x_i}, 0,t))\right)\leq 0 \ \ \text{on} \ \  \{0\} \times(0,T],\\[1mm]   
\max \left(\Sigma_i u_{i, x_i}-B, \max_i(u_{i,t} + H_i(u_{i, x_i}, 0,t))\right) \geq 0 \ \ \text{on} \ \  \{0\} \times(0,T],\\[1mm]
u(\cdot, 0)=u_0 \ \ \text{on} \ \  \overline  I.
\end{cases}
\end{equation}

\subsection*{Organization of the paper.} \no In the next section we state and prove an elementary lemma which is the basic tool for the proof of the comparison principle which is presented in Section~4. Section~5 is about the relation with the 
the flux limiters.

\section{A general lemma}

In this section we state and prove a general lemma which is the basic tool for the proof of Theorem \ref{main1}. It applies to problems of  one-dimensional junctions with Kirchoff condition and expands the class of ``gradients'' that can be used in the inequalities at the junction.
\begin{lemma}\label{lemma} Assume that $H_1, \ldots H_K \in C(\R)$, $p_1,\ldots,p_K, q_1,\dots, q_K \in \R$ and $a, b \in R$ are such that, for all $i\in \{1,\ldots,K\}$, 
%
%
 \begin{equation}\label{takis1000}
\begin{cases}
(i)~ ~~p_i\geq q_i \ \text{and} \ a + H_i(p_i) \leq 0 \leq b +H_i(q_i),\\[1mm]
(ii)~ ~\min \left(\Sigma_i p_i',  \min_i(a + H_i(p_i'))\right) \leq 0 \  \text{for all} \  p_i' \leq p_i,\\[1mm]
(iii) ~ \max \left(\Sigma_i q_i',  \max_i(b+ H_i(q_i')) \right)\geq0 \  \text{for all} \  q_i' \geq q_i.
\end{cases}
\end{equation}
\smallskip
%
%
%
Then $a\leq b$.
\end{lemma}
\begin{proof}
We argue by contradiction and assume that $a>b$. 
\smallskip

Modifying $p_1,\ldots,p_K, q_1,\dots, q_K, a$ and $b$ by small amounts and using the continuity of $H_1,\ldots,H_K$, we may assume that 
\begin{equation}\label{takis613}
 p_i>q_i \ \ \text{and} \ \  a + \max_iH_i(p_i) <0< b+  \min_iH_i(q_i).
 \end{equation}
If $\Sigma_i q_i \geq 0$, choose $\ep>0$ small enough so that $q_i +\ep/K <p_i$. Then $p_i'=q_i + \ep/K$ in \eqref{takis1000}(ii) yields 
$\min_i (a + H_i(q_i+\ep/K)) \leq 0,$ and, after letting $\ep\to 0$, $\min_i (a + H_i(q_i)) \leq 0,$ which, in view of \eqref{takis613},  is not possible given that it assumed that $a>b$.  
A similar argument yields a contradiction if  $\Sigma_i p_i \leq 0$.
\smallskip

Next we assume that $\Sigma_i p_i > 0 > \Sigma_i q_i,$  
and let $c\in (b,a)$ and $r_i \in (q_i, p_i)$ be such that $H_i(r_i)+c=0.$ If $\Sigma_ir_i \geq 0$ (resp. $\Sigma_ir_i \leq 0$), we choose $p_i'=r_i$ (resp. $q_i'=r_i$)   and argue as before. 
\qed
\end{proof}
\section{One-dimensional time dependent junctions}

Here we prove Theorem~\ref{main1} for the initial value problem \eqref{takis1}. The proof in the multi-dimensional setting is almost 
identical and we leave it up to the reader to fill in the details. The existence of solutions is immediate from Perron's method or Theorem~\ref{main2}. 
\smallskip

To simplify the presentation here we take 
$$B=0.$$

Although the proof is not long,  to clarify the strategy and highlight the new ideas, we present first a heuristic description of the argument assuming that $u_i,v_i \in C^1(\overline  I_i \times [0,T])$ with possible discontinuities in the spatial derivative as $i$ changes;  note that, since it also assumed that $u,v\in C(\overline I \times [0,T]),$ the previous assumption also gives that $u_t(0,t), v_t(0,t)$ exist for all $t\in (0,T].$ 
\smallskip

Following the proof of the classical maximum principle, we assume that,  for  $\delta>0$, 
the 

$\max_{x\in \overline I \times [0,T]} \left[ (u-v)(x,t) -\delta t\right]$ is achieved at  $(x_0,t_0) \in  \overline I \times [0,T]$ with  $t_0>0.$ If $x_0\neq 0$, we argue as in the classical uniqueness proof. Hence, we continue assuming that $x_0=0.$  
Let $a=u_t(0,t_0)$ and $b=v_t(0,t_0)$. It follows that  $a\geq b+\delta >b.$
\smallskip

The functions  $U(x_i)=u(x_i,t_0)$ and $V(x_i)=v(x_i,t_0)$ 
are smooth  sub-and super-(viscosity) solutions to 
\begin{equation}\label{takis1001}
\begin{cases}
 a+  H_i (U_{x_i},x_i,t_0)\leq 0 \ \text{in} \  \overline I_i  \ \ \text{and }  \ \
\min\left(\Sigma_iU_{x_i}(0), a+\min_i  H_i(U_{x_i}(0),0,t_0)\right)\leq 0,\\[1mm] 
 b+  H_i (V_{x_i},x_i,t_0)\geq 0 \ \text{in} \  \overline I_i  \ \ \text{and }  \ \
\max \left(\Sigma_i V_{x_i}(0), b+\max_i  H_i(V_{x_i}(0),0,t_0)\right) \geq 0,
\end{cases}
\end{equation}
while  $U(x_i)-V(x_i)\leq U(0)-V(0)$, which in turn implies that $U_{x_i}(0)\geq V_{x_i}(0)$.
\smallskip

We get a contradiction  using Lemma \ref{lemma} provided we verify that \eqref{takis1000} holds for  the obvious choices of $H_1,\ldots,H_K, p_1,\dots, p_K,q_1,\ldots, q_K.$
\smallskip
And this is  immediate since \eqref{takis1000}(i) is part of \eqref{takis1001}, while \eqref{takis1000}(ii),(iii) follow from the observation 
that $J^+U_i(0)=(-\infty, U_{x_i}(0)]$ and $J^-V_i(0)=[V_{x_i}(0), \infty)$ and the fact that inequalities must hold in the viscosity sense.
\smallskip

 We continue now with the actual proof of Theorem~\ref{main1} for \eqref{takis1}, which consists of making the above rigorous for $u,v \in \text{BUC}(\overline I\times [0,T]).$ For the beginning of the proof, we assume that the reader is 
familiar with the ``standard'' arguments of the proof of the comparison of viscosity solutions.  
\medskip

\begin{proof} Arguing  by contradiction we suppose that, for some $\delta>0$ and $t_0\in (0,T)$, 
\begin{equation}\label{takis541}
u(0,t_0)-v(0,t_0)-\delta t_0=\max_{(x,t)\in \overline I \times [0,T]} \left[ (u-v)(x,t) -\delta t\right].
\end{equation}

We leave it up to the reader to convince him/herself that, if the above does not hold as $\delta\to 0$, then the conclusion of the theorem holds.

\smallskip

Replacing $u$ and $v$ respectively by the classical $\sup$- and $\inf$-convolutions in time, we may assume that \eqref{takis541} still holds, perhaps for a different $t_0$ and shorter time interval,  and that $u$ and $v$ are respectively semiconvex and semiconcave and, hence, Lipschitz continuous with respect to $t$. Moreover, in view of the assumed uniform coercivity of the Hamiltonians, it follows that $u$ is also Lipschitz continuous in $x$; note that we do not make a similar claim for $v$. 
\smallskip

Of course, this means that we need to consider \eqref{takis1} in a smaller time interval and evaluate $H$ at a different time. It is, however, standard that this does not alter the outcome, and, hence, we omit the details.
\smallskip

In view of the assumed semiconvexity and semiconcavity of the $u$ and $v$ respectively, both of them are differentiable with respect to  $t$ at $(0,t_0)$. Let $a=u_t(0,t_0)$ and $b=v_t(0,t_0)$. It follows that 
$$a\geq b+\delta >b.$$
The next step  is an observation, which, heuristically speaking, establishes a $C^1$-type property for the sub- and super-differentials  of semiconvex and semiconcave functions near points of differentiability. Since the claim may be useful in other contexts, we state it  as a separate lemma. 
\begin{lemma}\label{lemma2} Let $z:(c,0]\times (0,T) \to \R $ be uniformly continuous in space, and semiconvex in time 
and assume that, for some $t_0\in (0,T)$, $\bar a=z_t(0,t_0)$ exists. If $\partial z(x,t)$ is the subdifferential of $z$ with respect to $t$ at $(x,t)$, then
\begin{equation}\label{takis3}
\lim_{(x,t)\to (0,t_0)}\sup_{p\in \partial z(x,t)}|p-\bar a| =0. 
\end{equation}
A similar statement is true for the superdifferential in $t$, if $z:(c,0]\times (0,T) \to \R $ is uniformly continuous in space, semiconcave in time and  differentiable at some $(0,t_0)$.
\end{lemma}
\smallskip
The claim follows from  the classical fact that, for  semiconvex functions, derivatives converge to derivatives, and the uniform continuity in $x$. 

\smallskip

In what follows, all claims and statements hold for $i=1,\ldots, K$.
\smallskip

Continuing the ongoing proof we observe that Lemma~\ref{lemma2} yields  $\eta_i:\overline I_i \times [0,T] \to \R$ such that 
$\lim_{(x_i,t)\to (0,t_0)}\eta_i(x_i,t)=0$ and, in the viscosity sense and in a neighborhood of $(0,t_0)$,
\begin{equation}\label{takis2}
a + H_i(u_{x_i},0,t_0) \leq  \eta_i(x_i,t)  \ \ \text{and} \ \  b + H_i(v_{x_i},0,t_0) \geq  \eta_i(x_i,t).
\end{equation}
Indeed, if $(p_i,\bar p_i) \in J^+u(x,t)$, then $\bar p_i \in \partial u(x,t)$, and the claim follows from the previous observations and the continuity properties of $H_i$. 
\smallskip

Next we use a blow up argument at $(0,t_0)$ on all branches to reduce the problem to a time independent setting to which we can apply Lemma \ref{lemma}. 
\smallskip

For $\ep>0$, let 
$$u^\ep_i(x_i,t)=\frac{u(\ep x_i, t_0 + \ep (t-t_0))-u(0,t_0)}{\ep} \ \ \text{and} \ \  v^\ep_i(x_i,t)=\frac{v(\ep x_i, t_0 + \ep (t-t_0))-v(0,t_0)}{\ep}.$$
It is immediate that $u^\ep_i$ and $v^\ep_i$ are respectively sub- and super-solution of \eqref{takis1}.
\smallskip

In view of the definition of $(0,t_0)$, the properties of $u$ and $v$ and the observations above, it is immediate that 
\begin{equation}\label{takis4}
u^\ep_i \leq v^\ep_i + \delta (t-t_0).
\end{equation}
Moreover, Lemma~\ref{lemma2} implies that,  as $\ep \to 0$ and locally uniformly in $(x,t)$,
\begin{equation}\label{takis4.1}
u^\ep_{i,t}(x_i,t)=u_t (\ep x_i, t_0 + \ep (t-t_0)) \to a \ \ \ \text{and} \ \  v^\ep_{i,t}(x_i,t)=v_t (\ep x_i, t_0 + \ep (t-t_0)) \to b. 
\end{equation}
Since, for $\tilde w=u^\ep_i$ or $\tilde w=v^\ep_i$ and $w=u_i$ or $w=v_i$, 
\[\tilde w (x_i,t)-\tilde w x_i,t_0)=\dfrac{w(\ep x_i,t_0+ \ep (t-t_0))-w(\ep x_i,t_0)}{\ep}=\int_0^1 w_{\tau}(\ep x_i+ \lambda \ep (t-t_0))d\lambda \  (t-t_0),\]
it follows, again from Lemma~\ref{lemma2} and \eqref{takis4.1}, that 
\begin{equation}\label{takis5}
u^\ep_{i}(x_i,t)-u^\ep_{i}(x_i,t_0) \to a(t-t_0) \ \ \text{and} \ \ 
v^\ep_{i}(x_i,t)-v^\ep_{i}(x_i,t_0) \to b(t-t_0).
\end{equation}
The space-time Lipschitz continuity of the $u_i$'s, \eqref{takis5} and the stability properties of sub-solutions imply that, along subsequences $\ep_n \to 0$,  
\[u_i^{\ep_n}(x_i,t) \to U_i(x_i) + a(t-t_0)\] and 
\begin{equation}\label{takis614}
a+ \overline  H(U_{i, x_i}) \leq 0,   \ \ \min \left(\Sigma_i U_{i, x_i}, \min_i(a + \overline H_i(U_{i, x_i}))\right) \leq 0 \ \ \text{and} \ \ U_i(0)=0, 
\end{equation}
where, for notational simplicity,  we write $\overline H_i(p)$ in place of $H_i(p, 0, t_0).$
\smallskip

Since the $v_i$'s are not necessarily Lipschitz continuous in space, it is not possible to use a similar argument for the 
$v^\ep_i$'s. Instead, we need to consider their  lower half-relaxed limit, which although finite from below, in view of \eqref{takis4}, may be infinite from above.

\smallskip

We go around this potential difficulty as follows. 
The coercivity of the Hamiltonians and the Lipschitz continuity of $u_i$ yield  $M>0$ such that
\[ b+H_i(M)\geq 0 \ \ \text{and} \ \ u^\ep_i(x_i,t_0) \leq Mx_i.\]

Hence $M x_i +b(t-t_0)$ and, therefore,
\[\tilde v^\ep_i(x_,t)= v^\ep_i(x,t) \wedge (Mx_i +b(t-t_0))\]
is a super-solution of \eqref{takis1}.
\smallskip

It follows from the choice of $M$ and \eqref{takis4}, \eqref{takis4.1} and \eqref{takis5} that the lower-half relaxed limit $V_{i,_\star}$ of the $\tilde v^\ep_i$ given by 
\[V_{i,_\star}(x_i,t)=\underset{\ep \to 0, (y_i,s) \to (x_i,t)} {\liminf} \tilde v^\ep(y_s,s), \]
is  finite from above and below, and that there exists a lower semicontinuous $V_i:\bar I_i \to \R$ such that 
\[V_{i,_\star}(x_i,t)=V_i(x_i)+b(t-t_0),  \ \ U_i\leq V_i  \ \text{in}  \ (-\infty,0] \ \  \text{and} \ \ V_{i}(0)=0.\]
The stability properties of the solutions also imply that 
%
%
%
%
%
\begin{equation}\label{takis6}
 b+ \overline H(V_{i, x_i}) \geq 0 \ \ \text{and} \ \ \max \left(\Sigma_i V_{i, x_i}, \max_i(b + \overline H_i(V_{i,x_i}))\right)\geq 0.
 \end{equation} 
Finally, we recall that $$a>b.$$
Next we get a contradiction using  Lemma \ref{lemma}. While the choice of the $H_i$'s is obvious, some work is necessary to identify  $p_1,\ldots,p_K,q_1,\ldots,p_K$ such that  \eqref{takis1000} holds. 
\smallskip

Set  $$\underbar p_i:= \liminf\limits_{x_i\to 0}\frac{U_{i}(x_i)}{x_i}, \ \  \overline p_i:=\limsup\limits_{x_i\to 0}\frac{U_{i}(x_i)}{x_i}, \ \ 
\underbar q_i:= \liminf\limits_{x_i\to 0}\frac{V_{i}(x_i)}{x_i} \ \ \text{ and} \ \   \overline q_i:=\limsup\limits_{x_i\to 0}\frac{V_{i}(x_i)}{x_i},$$
and recall that 
 $J^+U_i(0)=(-\infty, \underbar p_i]$ and  $J^-V_i(0)=[\overline q_i, \infty)$.
 \smallskip
 
Observe that  $U_i\leq V_i$ 
does not necessarily yield $\underbar p_i \geq \overline q_i,$ the latter being enough to conclude using  Lemma \ref{lemma} with $p_i=\underbar p_i$ and $q_i=\overline q_i$. Notice, however, that $U_i\leq V_i$  implies  
\begin{equation}\label{takis1002}
\overline p_i\geq \underbar q_i.
\end{equation}

Although $\overline p_i \notin J^+U_i(0)$ and $\underbar q_i \notin J^-V_i(0)$, unless $U_i$ and $V_i$ are respectively differentiable at $0$,  we claim that \eqref{takis1000} holds for $p_i=\overline p_i$ and $q_i=\underbar q_i.$


\smallskip

A classical blow-up argument (see, for example, Jensen and Souganidis \cite {JS}) shows that 
\begin{equation}\label{takis9}
a+ \overline H_i(p) \leq 0  \ \text{for all }  \  p\in [\underbar p_i, \overline p_i] \ \ \text{and} \ \  b+ \overline H_i(q) \geq 0  \ \text{for all } \  q\in [\underbar q_i, \overline q_i].
\end{equation}
Moreover, if $p_i'\leq \underbar p_i$ for all $i$, then $p'_i\in J^+U_i(0)$  and, hence, $\min \left(\Sigma_i p_i', \min_i(a+ \overline  H_i(p_i')) \right) \leq 0.$ 
\smallskip

If, for some fixed $i_{i_0}$, $p'_{i_0} \in [\underbar p_{i_0}, \overline p_{i_0}],$  then, in view of \eqref{takis9}, $a+ H_{i_0}(p'_{i_0}, 0, t_0) \leq 0$, and again  $\min \left(\Sigma_i p_i', \max_i(a+ \overline  H_i(p_i')) \right) \leq 0.$ It follows that  \eqref{takis1000}(ii) holds. 
\smallskip

A similar argument yields \eqref{takis1000}(iii), while \eqref{takis1000}(i) is obviously true, in view of \eqref{takis1002} and \eqref{takis9}. 
\end{proof}
\section{Flux-limiter solutions are generalized Kirchoff solutions}

We show here that the  flux-limiter solutions to  time-depending one dimension junction problems, which were introduced in \cite{IM}, are actually  generalized viscosity  solutions to \eqref{takis1} for an appropriate choice of $B$ in the Kirchoff-condition. 
\smallskip

We begin recalling the notion of flux-limiter solution. Following \cite{IM}, we assume that, for all $i=1,\ldots,K$, 
\begin{equation}\label{takis50}
\widetilde H_i  \in C(\R)  \ \text{is convex with a unique minimum at $p_i^0$},
\end{equation}
and define $\widetilde H_i^\pm:\R \to \R$ by
\begin{equation}\label{takis51}
\widetilde H_i^-(p)=\begin{cases} \widetilde H_i(p) \ \text{ if} \ p\leq p_i^0,\\[1mm]
\widetilde H_i(p_i^0) \ \text{ if} \ p\geq p_i^0,
\end{cases} \ \ \text{and} \ \ \ 
 \widetilde H_i^+(p)=\begin{cases} \widetilde H_i(p_i^0) \ \text{ if} \ p\leq p_i^0,\\[1mm]
\widetilde H_i(p) \ \text{ if} \ p\geq p_i^0;
\end{cases} 
\end{equation}
note that \cite {IM} considers quasiconvex $\widetilde H_i$'s but to simplify the presentation here we assume convexity. Finally, to simplify the presentation we assume that we deal with continuous solutions.
\smallskip

Fix  $A\geq A_0=\max_i\min_\R \widetilde H$ and let $\widetilde I_i=(0,\infty)$ and $\widetilde I=\bigcup_{i=1}^K \widetilde I_i$. Then  $\widetilde u\in BUC(\overline {\widetilde I} \times [0,T])$
is an  $A$-limiter solution of the junction problem if 
\begin{equation}\label{takis52}
\begin{cases}
\widetilde u_{i,t} + \widetilde H_i(\widetilde  u_{i,{x_i}}) =0 \ \ \text{ in} \ \  \widetilde I_i \times (0,T],\\[1mm]
\widetilde u_t + \max (A, \max_i \widetilde {H}^-_i(\widetilde  u_{i,{x_i}}))=0 \ \ \text{on}  \ \  \{0\}\times  (0,T].
\end{cases}
\end{equation} 
For each $i$, let $p_i^A$ be the unique solution to $\widetilde H_i(p)=A$ such that $p_i^A \geq p_i^0$, which exists in view of \eqref{takis51}.
\begin{proposition}
If $\widetilde u$ is an $A$-limiter solution, that it satisfies \eqref{takis52}, then $u:\overline I \to R$ defined by $u(x)=\widetilde u(-x)$
is a generalized solution to \eqref{takis1} for  $B= - \Sigma_{i=1}^K p_i^A$ and $H_i(p)=\widetilde H_i(-p).$
\end{proposition} 
 \begin{proof} The conclusion follows once we show that $\widetilde u$ is a solution to
 \begin{equation} \label{takis53}
 \begin{cases}
 \widetilde u_{i,t} + \widetilde H_i(\widetilde u_{i,{x_i}}) =0 \ \  \text{ in} \ \  \widetilde I_i \times (0,T],\\[1mm]
 \min\left( - \Sigma_{i=1}^K \widetilde u_{i,{x_i}} - B, \min_i( \widetilde u_{i,t} +\widetilde H_i(\widetilde u_{i,{x_i}})\right) \leq 0 \ \ \text{on} \ \  \{0\}\times (0,T],\\[1mm]
 \max\left( - \Sigma_{i=1}^K \widetilde u_{i,{x_i}} - B, \max_i( \widetilde u_{i,t} +\widetilde H_i(\widetilde u_{i,{x_i}})\right) \geq 0 \ \ \text{on} \ \  \{0\}\times (0,T].
 \end{cases}
 \end{equation}
Clearly we only need check the inequalities on $\{0\}\times (0,T].$ We begin with the sub-solution property and assume that, for some $t_0 \in (0,T]$ and  for each $i$, $(p_i, a) \in J^+u_i(0,t_0).$  
\smallskip

Since $\widetilde u$ is an $A$-limiter solution, for all $i$, we have 
\begin{equation}
a  + A=a+ \widetilde H_(p_i^A) \leq 0 \ \ \text{and} \ \ a+ {\widetilde H}^-_i(p_i) \leq 0.
\end{equation}
 \end{proof}
Assume that  
\begin{equation}\label{takis54}
 - \Sigma_{i=1}^K p_i + \Sigma_{i=1}^K p_i^A >0.
 \end{equation}
 It then follows that there exists $i_0$ such  $p_{i_0} < p_{i_0}^A$. Then, if $ p^0_{i_0}\leq p_{i_0}$, since we are in the increasing part of $\widetilde H_{i_0}$,  we have $\widetilde H_{i_0}(p_{i_0}) \leq \widetilde H_{i_0}(p_{i_0}^A)=A$, and, hence
 \begin{equation}\label{takis55}
 a + \widetilde H_{i_0}(p_{i_0}) \leq 0.
 \end{equation}
If  $ p^0_{i_0}\geq p_{i_0},$ then $\widetilde H_{i_0} (p_{i_0})=\widetilde { H}^-_{i_0} (p_{i_0}),$ and again we have  \eqref{takis55},
and, hence, the sub-solution property.
\smallskip

For  the super-solution property we assume that, for some $t_0 \in (0,T]$ and  for each $i$, $(q_i, a) \in J^-u_i(0,t_0).$  It then follows from the definition of the $A$-limiter solution that 
\begin{equation}\label{takis56}
a + \max ( A, \max_i \widetilde H^-_{i}(q_{i}) )\geq 0.
\end{equation}
If  $\max_i \widetilde H^-_{i}(q_{i})\geq A$, then, since $ \widetilde H_{i}(q_{i}) \geq  \widetilde H^-_{i}(q_{i})$, 
\begin{equation}\label{takis57}
\max ( - \Sigma_{i=1}^K q_i - B, \max_i( a +\widetilde H_i(q_i))) \geq 0.
\end{equation} 
If  $A>\max_i \widetilde H^-_{i}(q_{i})$, then, for all $i$, 
\begin{equation}\label{takis58}
a+ \widetilde H^+_{i}(p_i^A)=a+A \geq 0.
\end{equation} 
Assume that $ - \Sigma_{i=1}^K q_i + \Sigma_{i=1}^K p_i^A\leq 0$, for otherwise \eqref{takis57} is satisfied.
\smallskip

Then there must exist some $i_0$ such that $q_{i_0} \geq p_{i_0}^A$, which implies that $\widetilde H_{i_0}(q_{i_0})\geq  \widetilde H^+_{i_0}(p_{i_0}^A)=A$, and \eqref{takis57} holds again.
\smallskip

The claim then follows using that $u(x)=\widetilde u(-x).$ \qed



\noindent ($^{1}$) College de France,
11 Place Marcelin Berthelot, 75005 Paris, 
and  
CEREMADE, 
Universit\'e de Paris-Dauphine,
Place du Mar\'echal de Lattre de Tassigny,
75016 Paris, FRANCE\\ 
email: lions@ceremade.dauphine.fr
\\ \\
\noindent ($^{2}$) Department of Mathematics 
University of Chicago, 
5734 S. University Ave.,
Chicago, IL 60637, USA\\ 
email: souganidis@math.uchicago.edu
\\ \\
($^{3}$)  Partially supported by the National Science Foundation and the Office for Naval Research.

\end{document}